\documentclass[11 pt]{amsart}
\usepackage{amsmath,enumerate,amsthm,amssymb,amscd}

\DeclareMathOperator{\Supp}{Supp}

\newcommand{\Spec}{\operatorname{Spec}}

\newcommand{\Hom}{\operatorname{Hom{}}}

\renewcommand{\phi}{\varphi}
\renewcommand{\to}{{\longrightarrow}}

\newcommand{\Ann}{\operatorname{Ann}}

\newtheorem{thm}{Theorem}[section]

\newtheorem{prop}[thm]{Proposition}
\newtheorem{lemma}[thm]{Lemma}

\newtheorem{example}[thm]{Example}

\numberwithin{equation}{section}

\begin{document}
\title{A change of rings result for Matlis reflexivity}
\author{Douglas J. Dailey and Thomas Marley}
\address{Department of
Mathematics\\
University of Nebraska-Lincoln\\Lincoln,  NE 68588-0130}
\email{ ddailey2@math.unl.edu}
\address{Department of
Mathematics\\
University of Nebraska-Lincoln\\Lincoln,  NE 68588-0130}
\email{ tmarley1@unl.edu}


\begin{abstract} Let $R$ be a commutative Noetherian ring and $E$ the minimal injective cogenerator of the category of $R$-modules.   An $R$-module $M$ is (Matlis) reflexive if the natural evaluation map $M\to \Hom_R(\Hom_R(M,E),E)$ is an isomorphism.  We prove that if $S$ is a multiplicatively closed subset of $R$ and $M$ is a reflexive $R$-module, then $M$ is a reflexive $R_S$-module.  The converse holds when $S$ is the complement of the union of finitely many nonminimal primes of $R$, but fails in general.

\end{abstract}

\date{October 14, 2015}

\bibliographystyle{amsplain}

\thanks{The first author was partially supported by U.S. Department of Education grant P00A120068 (GAANN)}

\maketitle

\begin{section}{Introduction}
\end{section}

Let $R$ be a commutative Noetherian ring and $E$ the minimal injective cogenerator of the category of $R$-modules; i.e., $E=\bigoplus_{m\in \Lambda} E_R(R/m)$, where $\Lambda$ denotes the set of maximal ideals of $R$ and $E_R(-)$ denotes the injective hull.    An  $R$-module $M$ is said to be {\it (Matlis) reflexive}  if the natural evaluation map $M\to \Hom_R(\Hom_R(M,E), E)$ is an isomorphism.   In \cite{BER},  the authors assert the following ``change of rings''
principal for Matlis reflexivity (\cite[Lemma 2]{BER}): {\it Let $S$ be a multiplicatively closed subset of $R$ and suppose $M$ is an $R_S$-module.  Then $M$ is reflexive as an $R$-module if and only if $M$ is reflexive as an $R_S$-module.}   However,  the proof given in \cite{BER} is incorrect (see Examples 3.1-3.3)  and in fact the ``if" part is false in general (cf. Proposition \ref{counter}).  In this note, we prove the following:

\medskip

\begin{thm} \label{main} Let $R$ be a Noetherian ring, $S$ a multiplicatively closed subset of $R$, and $M$ an $R_S$-module.
\begin{enumerate}[(a)]
\item If $M$ is reflexive as an $R$-module then $M$ is reflexive as an $R_S$-module.
\item If $S=R\setminus (p_1\cup \ldots \cup p_r)$ where each $p_i$ is a maximal ideal or a nonminimal prime ideal, then the converse to (a) holds.
\end{enumerate}
\end{thm}

\begin{section}{Main results}
\end{section}

Throughout this section $R$ will denote a Noetherian ring and $S$ a multiplicatively closed set of $R$.
We let $E_R$  (or just $E$ if the ring is clear) denote the minimal injective cogenerator of the category of $R$-modules as defined in the introduction.    A semilocal ring is said to be complete if it is complete with respect to the $J$-adic topology, where $J$ is the Jacobson radical.

We will make use of the main result of \cite{BER}:

\begin{thm}{\rm(\cite[Theorem 12]{BER})} \label{ber}  Let $R$ be a Noetherian ring, $M$ an $R$-module, and $I=\Ann_RM$.  Then $M$ is reflexive if and only if $R/I$ is a complete semilocal ring and there exists a finitely generated submodule $N$ of $M$ such that $M/N$ is Artinian.
\end{thm}

We remark that the validity of this theorem does not depend on \cite[Lemma 2]{BER}, as the proof of \cite[Theorem 12]{BER} uses this lemma only in a special case where it is easily seen to hold. (See the proof of \cite[Theorem 9]{BER}, which is the only instance \cite[Lemma 2]{BER} is used critically.)

\begin{lemma}{\rm(\cite[Lemma 1]{BER}} \label{lem1} Let $M$ be an $R$-module and $I$ an ideal of $R$ such that $IM=0$.  Then $M$ is reflexive as an $R$-module if and only if $M$ is reflexive as an $R/I$-module.  
\end{lemma}
\begin{proof} Since $E_{R/I}=\Hom_R(R/I, E_R)$,  the result follows readily by Hom-tensor adjunction.
\end{proof}

\begin{lemma} \label{prod} Let $R=R_1\times \cdots \times R_k$ be a product of Noetherian local rings.   Let $M=M_1\times\cdots \times M_k$ be an $R$-module.   Then $M$ is reflexive as an $R$-module if and only if $M_i$ is reflexive as an $R_i$-module for all $i$.
\end{lemma}
\begin{proof} Let $\rho_i:R\to R_i$ be the canonical projections for $i=1,\dots,k$.  Let $n_i$ be the maximal ideal of $R_i$ and $m_i=\rho_i^{-1}(n_i)$ the corresponding maximal ideal of $R$.  Then $m_iR_i=n_i$ and $m_iR_j=R_j$ for all $j\neq i$.  Note that $R_{m_i}\cong R_i$ and  $E_i:=E_R(R/m_i)\cong E_{R_i}(R_i/n_i)$ for all $i$.  Then $E_R=E_1\oplus \cdots \oplus E_k$.  It is easily seen that 
$$\Hom_R(\Hom_R(M,E_R),E_R)\cong \bigoplus_{i=1}^k\Hom_{R_i}(\Hom_{R_i}(M_i, E_i), E_i),$$
and that this isomorphism commutes with the evaluation maps.  The result now follows.
\end{proof}

\begin{thm} Let $S$ be a multiplicatively closed set of $R$ and $M$ an $R_S$-module which is reflexive as an $R$-module.  Then $M$ is reflexive as an $R_S$-module.
\end{thm}
\begin{proof} By Lemma \ref{lem1}, we may assume $\Ann_R M=\Ann_{R_S}M=0$.  Thus, $R$ is semilocal and complete by Theorem \ref{ber}.  Hence, $R=R_1\times\cdots\times R_k$ where each $R_i$ is a complete local ring.  Then $R_S=(R_1)_{S_1}\times \cdots \times (R_k)_{S_k}$ where $S_i$ is the image of $S$ under the canonical projection $R\to R_i$.   Write $M=M_1\times \cdots \times M_k$, where $M_i=R_iM$.   As $M$ is reflexive as an $R$-module, $M_i$ is reflexive as an $R_i$-module for all $i$.  Thus, it suffices to show that $M_i$ is reflexive as an $(R_i)_{S_i}$-module for all $i$.   Hence, we may reduce the proof to the case $(R,m)$ is a complete local ring with $\Ann_R M=0$ by passing to $R/\Ann_RM$, if necessary.   As $M$ is reflexive as an $R$-module, we have by Theorem \ref{ber}
that there exists an exact sequence
$$0\to N\to M\to X\to 0$$
where $N$ is a finitely generated $R$-module and $X$ is an Artinian $R$-module.   If $S\cap m=\emptyset$, the $R_S=R$ and there is nothing to prove.  Otherwise, as $\Supp_R X\subseteq \{m\}$, we have $X_S=0$.   Hence, $M\cong N_S$, a finitely generated $R_S$-module.  To see that $M$ is $R_S$-reflexive, it suffices to show that $R_S$ is Artinian (hence semilocal and complete).  Since $\Ann_R N_S=\Ann_R M=0$, we have that $\Ann_R N=0$.  Thus, $\dim R=\dim N$.   Since $M$ is an $R_S$-module and $S\cap m\neq \emptyset$, we have $H^i_m(M)\cong H^i_{mR_S}(M)=0$ for all $i$.  Further, as $X$ is Artinian, $H^i_m(X)=0$ for $i\ge 1$.  Thus, from the long exact sequence on local cohomology, we conclude that $H^i_m(N)=0$ for $i\ge 2$.  Thus, $\dim R=\dim N\le 1$, and hence, $\dim R_S=0$.  Consequently, $R_S$ is Artinian, and $M$ is a reflexive $R_S$-module.
\end{proof}

\medskip

To prove part (b) of Theorem \ref{main}, we will need the following result on Henselian local rings found in \cite{BKKN} (in which the authors credit it to F. Schmidt).  As we need a slightly different version of this result than what is stated in \cite{BKKN} and the proof is short, we include it for the convenience of the reader:

\begin{prop} \label{hensel} {\rm (\cite[Satz 2.3.11]{BKKN})}  Let $(R,m)$ be a local Henselian domain which is not a field and $F$ the field of fractions of $R$.  Let $V$ be a discrete valuation ring with field of fractions $F$.  Then $R\subseteq V$.
\end{prop}
\begin{proof} Let $k$ be the residue field of $R$ and $a\in m$.  As $R$ is Henselian, for every positive integer $n$ not divisible by the characteristic of $k$, the polynomial $x^n-(1+a)$ has a root $b$ in $R$.  Let $v$ be the valuation on $F$ associated to $V$.  Then $nv(b)=v(1+a)$.  If $v(a)<0$ then $v(1+a)<0$ which implies $v(b)\le -1$.  Hence, $v(1+a)\le -n$.  As $n$ can be arbitrarily large, this leads to a contradiction.  Hence, $v(a)\ge 0$ and $a\in V$.  Thus, $m\subseteq V$.  Now let $c\in R$ be arbitrary.  Choose $d\in m, d\neq 0$.  If $v(c)<0$ then $v(c^{\ell}d)<0$ for $\ell$ sufficiently large.
But this contradicts that $c^{\ell}d\in m\subseteq V$ for every $\ell$.  Hence $v(c)\ge 0$ and $R\subseteq V$.
\end{proof}

For a Noetherian ring $R$, let $\operatorname{Min}R$ and $\operatorname{Max}R$ denote the set of minimal and maximal primes of $R$, respectively.  Let $\operatorname{T}(R)=(\Spec R\setminus \operatorname{Min}R)\cup \operatorname{Max}R$.

\begin{lemma} \label{local} Let $R$ be a Noetherian ring and $p\in \operatorname{T}(R)$.  If $R_p$ is Henselian then the natural map $\phi:R\to R_p$ is surjective; i.e., $R/\ker \phi\cong R_p$.
\end{lemma}
\begin{proof} By replacing $R$ with $R/\ker \phi$, we may assume $\phi$ is injective.  Then $p$ contains every minimal prime of $R$.  Let $u\in R, u\not\in p$.  It suffices to prove that the image of $u$ in $R/q$ is a unit for every minimal prime $q$ of $R$.  Hence, we may assume that $R$ is a domain.  (Note that $(R/q)_p=R_p/qR_p$ is still Henselian.)  If $R_p$ is a field, then, as $p\in \operatorname{T}(R)$, we must have $R$ is a field (as $p$ must be both minimal and maximal in a domain).  So certainly $u\not\in p=(0)$ is a unit in $R$. Thus, we may assume $R_p$ is not a field.  Suppose $u$ is not a unit in $R$.  Then $u\in n$ for some maximal ideal $n$ of $R$.  Now, there exists a discrete valuation ring $V$
with same field of fractions as $R$ such that $m_V\cap R=n$ (\cite[Theorem 6.3.3]{SH}). As $R_p$ is Henselian, $R_p\subseteq V$ by Proposition \ref{hensel}.  But as $u\notin p$, $u$ is a unit in $R_p$, hence in $V$, contradicting $u\in n\subseteq m_V$.  Thus, $u$ is a unit in $R$ and $R=R_p$.
\end{proof}

\begin{prop} \label{complete} Let $R$ be a Noetherian ring and $S=R\setminus (p_1\cup \cdots \cup p_r)$ where $p_1,\dots, p_r \in \operatorname{T}(R)$.  Suppose $R_S$ is complete with respect to its Jacobson radical.
Then the natural map $\phi: R\to R_S$ is surjective.  
\end{prop}
\begin{proof} First, we may assume that $p_j \nsubseteq \bigcup_{i\neq j}p_i$ for all $j$. Also, by passing to the ring $R/\ker \phi$, we may assume $\phi$ is injective. (We note that if $p_{i_1},\dots,p_{i_t}$ are the ideals in the set $\{p_1,\dots, p_r\}$ containing $\ker \phi$, it is easily seen that  $(R/\ker \phi)_S=(R/\ker \phi)_T$ where $T=R\setminus (p_{i_1}\cup \cdots \cup p_{i_t})$.  Hence, we may assume each $p_i$ contains $\ker \phi$.)   As $R_S$ is semilocal and complete, the map $\psi:R_S\to R_{p_1}\times\cdots \times R_{p_r}$ given by $\psi(u)=(\frac{u}{1},\dots, \frac{u}{1})$ is an isomorphism.  For each $i$, let $\rho_i:R\to R_{p_i}$ be the natural map.  Since $R\to R_S$ is an injection, $\cap_i \ker \rho_i=(0)$.  It suffices to prove that $u$ is a unit in $R$ for every $u\in S$.  As $R_{p_i}$ is complete, hence Henselian, we have that $\rho_i$ is surjective for each $i$ by Lemma \ref{local}.   Thus, $u$ is a unit in $R/\ker \rho_i$ for every $i$; i.e., $(u)+\ker \rho_i=R$ for $i=1,\dots, r$.  Then $(u)=(u)\cap (\cap_i \ker \rho_i)=R$.  Hence, $u$ is a unit in $R$. 
\end{proof}

We now prove part (b) of the Theorem \ref{main}:

\begin{thm} \label{partB} Let $R$ be a Noetherian ring and $M$ a reflexive $R_S$-module, where $S$ is the complement in $R$ of the union of finitely many elements of $\operatorname{T}(R)$.  Then $M$ is reflexive as an $R$-module.
\end{thm}
\begin{proof} We may assume $M\neq 0$.  Let $S=R\setminus (p_1\cup \cdots \cup p_r)$, where $p_1,\dots, p_r \in \operatorname{T}(R)$ Let $I=\operatorname{Ann}_{R} M$, whence $I_S=\Ann_{R_S} M$.  
As in the proof of Proposition \ref{complete}, we may assume each $p_i$ contains $I$.  Then by Lemma \ref{lem1}, we may reduce to the case $\Ann_R M=\Ann_{R_S} M=0$.   Note that this implies the natural map $R\to R_S$ is injective.   As $M$ is $R_S$-reflexive, $R_S$ is complete with respect to its Jacobson radical by Theorem \ref{ber}.  By Proposition \ref{complete}, we have that $R\cong R_S$ and hence $M$ is $R$-reflexive.
\end{proof}


\begin{section}{Examples}
\end{section}

The following examples show that $\Hom_R(R_S,E_R)$ need not be the minimal injective cogenerator for the category of $R_S$-modules, contrary to what is stated in the proof of \cite[Lemma 2]{BER}:
\medskip

\begin{example}{\rm Let $(R,m)$ be a local ring of dimension at least two and $p$ any prime which is not maximal or minimal.  By \cite[Lemma 4.1]{MS}, every element of $\Spec R_p$ is an associated prime of the $R_p$-module $\Hom_R(R_p,E_R)$.  In particular, $\Hom_R(R_p,E_R)\not\cong E_{R_p}$.}
\end{example}

\begin{example}{\rm (\cite[p. 127]{MS}) Let $R$ be a local domain such that the completion of $R$ has a nonminimal prime contracting to $(0)$ in $R$.  Let $Q$ be the field of fractions of $R$.  Then $\Hom_R(Q,E_R)$ is not Artinian.}
\end{example}

\begin{example}{\rm Let $R$ be a Noetherian domain which is not local.  Let $m\neq n$ be maximal ideals of $R$.   By a slight modification of the proof of \cite[Lemma 4.1]{MS}, one obtains that $(0)$ is an associated prime of $\Hom_R(R_m,E_R(R/n))$, which is a direct summand of $\Hom_R(R_m, E_R)$.  Hence, $\Hom_R(R_m,E_R)\not\cong E_{R_m}$.}
\end{example}

We now show that the converse to part (a) of Theorem \ref{main} does not hold in general.  Let $R$ be a domain and $Q$ its field of fractions.  Of course, $Q$ is reflexive as a $Q=R_{(0)}$-module.
But as the following theorem shows, $Q$ is rarely a reflexive $R$-module.

\begin{prop} \label{counter} Let $R$ be a Noetherian domain and $Q$ the field of fractions of $R$.  Then $Q$ is a reflexive $R$-module if and only if $R$ is a complete local domain of dimension at most one.
\end{prop}  
\begin{proof} We first suppose $R$ is a one-dimensional complete local domain with maximal ideal $m$.  Let $E=E_R(R/m)$.  By \cite[Theorem 2.5]{Sch}, $\Hom_R(Q, E)\cong Q$.  Since the evaluation map of the Matlis double dual is always injective, we obtain that $Q\to \Hom_R(\Hom_R(Q, E), E)$ is an isomorphism.  

Conversely, suppose $Q$ is a reflexive $R$-module.  By Theorem \ref{ber}, $R$ is a complete semilocal domain, hence local.  It suffices to prove that $\dim R\le 1$.  Again by Theorem \ref{ber}, there exists a finitely generated $R$-submodule $N$ of $Q$ such that $Q/N$ is Artinian.  Since $\Ann_R N=0$, $\dim R=\dim N$.  Thus, it suffices to prove that $H^i_m(N)=0$ for $i\ge 2$.   But this follows readily from the facts that $H^i_m(Q)=0$ for all $i$ and $H^i_m(Q/N)=0$ for $i\ge 1$ (as $Q/N$ is Artinian). 
\end{proof}

\medskip

\noindent {\bf Acknowledgments:}  The authors would like to thank Peder Thompson for many helpful discussions on this topic.  They are also very grateful to Bill Heinzer for pointing out the existence of Proposition \ref{hensel}.

\end{document}